\documentclass[12pt,twoside,reqno]{amsart}
\usepackage{pifont}
\usepackage{caption}
\usepackage{amsthm}
\usepackage{amsfonts}
\usepackage{amssymb}
\usepackage{latexsym}
\usepackage{amstext}
\usepackage{array}
\usepackage{graphicx}
\usepackage{multirow}
\usepackage{amsmath}
\usepackage{enumerate}
  
\date{}
\allowdisplaybreaks[4] \footskip=15pt
\renewcommand{\uppercasenonmath}[1]{}

\newtheorem{thm}[subsection]{Theorem}
\newtheorem{cor}[subsection]{Corollary }
\newtheorem{Def}[subsection]{Definition}
\newtheorem{lem}[subsection]{Lemma}
\newtheorem{remark}[subsection]{Remark}

\newtheorem{prop}[subsection]{Proposition}
\newtheorem{exm}[subsection]{Example}
\newtheorem{conj}[subsection]{Conjecture}
  
\newcommand{\bthm}{\begin{thm} }
\newcommand{\ethm}{\end{thm} }
\newcommand{\bconj}{\begin{conj} }
\newcommand{\econj}{\end{conj} }
\newcommand{\bpro}{\begin{prop}}
\newcommand{\epro}{\end{prop}}
\newcommand{\bdf}{\begin{Def}}
\newcommand{\edf}{\end{Def}}
\newcommand{\bexm}{\begin{exm}}
\newcommand{\eexm}{\end{exm}}
\newcommand{\blem}{\begin{lem}}
\newcommand{\elem}{\end{lem}}
\newcommand{\bpf}{\begin{proof}}
\newcommand{\epf}{\end{proof}}
\newcommand{\bcor}{\begin{cor}}
\newcommand{\ecor}{\end{cor}}
\newcommand{\ba}{\begin{array}}
\newcommand{\ea}{\end{array}}
\newcommand{\bea}{\begin{eqnarray}}
\newcommand{\eea}{\end{eqnarray}}

\newcommand{\brem}{\begin{remark}}
\newcommand{\erem}{\end{remark}}
\newcommand{\nn}{\nonumber}
\newcommand{\ie}{\emph{i.e.}}

\def\defeq{{\,\stackrel{\Delta}{=}}\,}
\def\eg{{\it e.g.,\ \/}}

\begin{document}
\begin{center}
{{\large  \bf  Generalization and Alternative Proof of Two Identities Posed by Sun}}\\[0.2em]

 {\small  Keqin Liu}

{ \small School of Mathematics and Physics, Xi'an Jiaotong-Liverpool University, China\\Jiangsu National Center for Applied Mathematics, China\\keqin.liu@xjtlu.edu.cn}\\

\end{center}

\textbf{Abstract.} We study two identities involving roots of unity and determinants of Hermitian matrices which have been recently proved by using the famous eigenvector-eigenvalue identity for normal matrices. In this paper, we extend these identities to a more general form by considering the class of circulant matrices. Furthermore, we give an alternative proof of Sun's identities independent of the eigenvector-eigenvalue identity, where our strategy is built upon the similarity of an unnecessarily normal matrix to a particular matrix with integer eigenvalues, derived from the Fourier transform vectors.

\textbf{Key words:\hspace{0.1cm}} trigonometric identities; circulant matrices; permutations of integer eigenvalues; Fourier vectors

\textbf{MSC2010 subject classification:\hspace{0.1cm}} 11C20; 15A18; 42A16; 11A07; 05A19

\section{Introduction}
In 2018, Zhi-Wei Sun proposed a system of open problems in number theory and formalized numerous conjectural identities with meaningful arithmetic properties \cite{Sun1,Sun2,Sun3,Sun4,Sun5}. Identities~\eqref{eqn:odd1} and~\eqref{eqn:odd2} about roots of unity and derangements were conjectured by Sun in \cite{Sun4,Sun6}. Recently, Guo {\em et al.} proved \eqref{eqn:odd1} in \cite{Guo} based on the famous eigenvector-eigenvalue identity~\eqref{eqn:EigenEqual} for normal matrices \cite{DentonEtal}. With the same technique, Wang and Sun proved \eqref{eqn:odd2} in \cite{Sun6}.

Let $D(n)$ denote the set of all derangements $\tau$ of indices $j=1,\cdots,n$ such that $\tau(j)\neq j$ for all $j=1,\cdots,n$. Let~$\zeta$ a primitive $n$-th root of unity in the complex field $\mathbb{C}$.

\begin{thm}\label{thm:1}{(Trigonometric Identity-1)}
Let $n>1$ be an odd number. Then
\begin{eqnarray}\label{eqn:odd1}
&&\sum_{\tau\in D(n-1)}\mathrm{sign}(\tau)\prod_{j=1}^{n-1}\frac{1}{1-\zeta^{j-\tau(j)}}=
\frac{(-1)^{\frac{n-1}{2}}}{n}\left(\frac{n-1}{2}!\right)^2.
\end{eqnarray}
\end{thm}

\begin{thm}\label{thm:2}{(Trigonometric Identity-2)}
Let $n>1$ be an odd number. Then
\begin{eqnarray}\label{eqn:odd2}
&&\sum_{\tau\in D(n-1)}\mathrm{sign}(\tau)\prod_{j=1}^{n-1}\frac{1+\zeta^{j-\tau(j)}}{1-\zeta^{j-\tau(j)}}=
(-1)^{\frac{n-1}{2}}\frac{((n-2)!!)^2}{n}.
\end{eqnarray}
\end{thm}

In this paper, we generalize Theorems~\ref{thm:1} and~\ref{thm:2} to a more general form based on circulant matrices. Furthermore, we give a different technique for proving the above identities without using the eigenvector-eigenvalue identity~\eqref{eqn:EigenEqual} and solve for the integer eigenvalues of some related non-Hermitian (and unnecessarily normal) matrices.

\section{Generalization by Circulant Matrices}

Fix an integer~$n>1$. Let $f:\mathbb{Z}\rightarrow\mathbb{C}$ be a function satisfying the following conditions:
\begin{itemize}
\item[(i)] $f(i)=f(j)$ if $i\equiv j~\pmod{n}$.
\item[(ii)] There exist an $s\in\mathbb{Z}$ and an $n$-th root of unity $\zeta\in\mathbb{C}$ such that $$\sum_{k=0}^{n-1}f(k)\zeta^{-ks}=0.$$
\item[(iii)] For any $i,j\in\{1,2,\cdots,n\}$, $$\sum_{k=1}^{n}f(i-k)\overline{f(j-k)}=\sum_{k=1}^{n}\overline{f(k-i)}f(k-j),$$
where $\overline{a}$ denotes the complex conjugate of any $a\in\mathbb{C}$.
\end{itemize}

In other words, $f$ mapped from the rational integer domain is assumed to be a periodic function with period $n$ and has at least one coefficient of its discrete Fourier transform equal to zero. Furthermore, we require $f$ to have a symmetric property (condition (iii) above) which is clearly satisfied if $f(k)=\overline{f(-k)}$ for any integer~$k$. Let $M=(m_{ij})_{1\le i,j\le n}$ denote the $n\times n$ matrix with $m_{ij}=f(i-j)$, \ie, $M$ be a circulant matrix given by $f$. Then we have the following theorem about the determinant of the $(n-1)\times(n-1)$ minor $M_j$ of $M$.

\begin{thm}\label{thm:det}{(Determinants of Minors of Circulant Matrices)}
Let $n>1$ be an integer and $f$ a function satisfying conditions~(i-iii) above. Then for any minor $M_j$ of $M=(f(i-k))_{1\le i,k\le n}$ obtained by deleting the $j$-th row and column from the
matrix $M$, we have
\begin{eqnarray}\label{eqn:minor}
\det{M_j}=\frac{1}{n}\prod_{i=1,i\neq s}^n\lambda_i(M),
\end{eqnarray}
where $(\lambda_i(M))_{1\le i\le n}$ are the eigenvalues of $M$ given by
\begin{eqnarray}\label{eqn:eigen}
\lambda_i(M)=\sum_{k=0}^{n-1}f(k)\zeta^{-ki}.
\end{eqnarray}
\end{thm}

\begin{proof}
Since $M$ is circulant (condition (i) above), we have that for any column vector $$v_i=\frac{1}{\sqrt{n}}(\zeta^{i},\zeta^{2i},\cdots,\zeta^{ni})'$$ and any $k\in\{1,2,\cdots,n\}$ \cite{Petrov},
\begin{eqnarray}\label{eqn:eigenVector}
\frac{1}{\sqrt{n}}\sum_{j=1}^n f(k-j)\zeta^{ji}=\frac{1}{\sqrt{n}}\sum_{j=0}^{n-1} f(j)\zeta^{-ji}\zeta^{ki}.
\end{eqnarray}
So $\sum_{j=0}^{n-1}f(j)\zeta^{-ji}$ is an eigenvalue of $M$ whose eigenvectors $(v_i)_{1\le i\le n}$ form an orthonormal basis of $\mathbb{C}^n$.

By condition (iii), $MM^*=M^*M$ and $M$ is normal. We can thus apply the eigenvector-eigenvalue equality~\eqref{eqn:EigenEqual} to $M$. Now the proof of~\eqref{eqn:minor} is completed by choosing $\lambda_i(M)=0$ which exists by condition~(ii).
\begin{thm}\label{EigenIdentity}{(Eigenvector-eigenvalue identity~\cite{DentonEtal})}
Let~$M$ be an $n\times n$ normal matrix and~$M_j$ its minor obtained by deleting the $j$-th row and column from the matrix $M$. Let $\lambda_i(A)$ denote the $i$-th eigenvalue of any matrix~$A$ and~$v_i=(v_{ij})_{1\le j\le n}$ a unit eigenvector of~$M$ associated to $\lambda_i(M)$. Then we have
\begin{eqnarray}\label{eqn:EigenEqual}
|v_{ij}|^2\prod_{k=1;k\neq i}^n(\lambda_i(M)-\lambda_k(M))=\prod_{k=1}^{n-1}(\lambda_i(M)-\lambda_k(M_j)).
\end{eqnarray}
\end{thm}

\end{proof}

Now we prove that Theorems~\ref{thm:1} and~\ref{thm:2} are special cases of Theorem~\ref{thm:det} with proper selections of $f$ satisfying conditions~(i-iii). Specifically, we choose the following~$f$ for Theorems~\ref{thm:1}:
\begin{eqnarray}
f(k)=\begin{cases}
		\frac{1}{1-\zeta^k}, & k\in\{1,\cdots,n-1\} \\
		0, & k=0
	\end{cases}.
\end{eqnarray}
By~\cite{CalogeroPerelomov} or~\eqref{eqn:eigenV1} in~Lemma~\ref{lem:lm1} in the appendix, we have  $$\sum_{k=0}^{n-1}f(k)\zeta^{\frac{k(n+1)}{2}}=0.$$ So conition (ii) is satisfied. Conditions (i) and (iii) are clearly satisfied since $f(k)=\overline{f(-k)}$. Using~\eqref{eqn:eigen} and~\eqref{eqn:eigenV1} in~Lemma~\ref{lem:lm1} again we obtain that the eigenvalues of $M$ are $\{-\frac{n-1}{2},\cdots,-1,0,1,\cdots,\frac{n-1}{2}\}$. Choosing $j=n$ and $s=-\frac{(n+1)}{2}$ finishes the proof of Theorems~\ref{thm:1} immediately since $\det(M_n)$ equals to the left-hand side of~\eqref{eqn:odd1}. The proof of Theorem~\ref{thm:2} follows similarly by choosing $j=n$, $s=0$ and
\begin{eqnarray}
f(k)=\begin{cases}
		\frac{1+\zeta^k}{1-\zeta^k}, & k\in\{1,\cdots,n-1\} \\
		0, & k=0
	\end{cases}.
\end{eqnarray}
Note that now $M$ has eigenvalues $\{-(n-2),-(n-4),\cdots,-3,-1,0,1,3,\cdots,n-4,n-2\}$ by~\eqref{eqn:eigenV1} in~Lemma~\ref{lem:lm1}.

A more general example is given by
\begin{eqnarray}\label{eqn:findABC}
f(k)=\begin{cases}
		\frac{a+b\zeta^{ck}}{1-\zeta^k}, & k\in\{1,\cdots,n-1\} \\
		0, & k=0
	\end{cases},
\end{eqnarray}
where $a,b,c\in\mathbb{Z}$ are chosen to satisfy condition (ii) and recall that $\zeta\in\mathbb{C}$ is a {\it primitive} $n$-th root of unity.  Conditions (i) and (iii) are clearly satisfied since $f(k)=\overline{f(-k)}$. To find $a,b,c$ satisfying condition (ii), we investigate the solvability of the following equation on $s$:
\begin{eqnarray}\label{eqn:abcEquiv}
a\{s\}+b\{c+s\}=(a+b)\frac{n+1}{2},
\end{eqnarray}
where
\begin{eqnarray}
\{s\}:=\begin{cases}
		k, & k\in\{1,\cdots,n-1\} ~~\&~~ s\equiv k \pmod n \\
		n, & s\equiv 0 \pmod n 
	\end{cases}.
\end{eqnarray}
From~\eqref{eqn:eigenV1} in~Lemma~\ref{lem:lm1}, if~\eqref{eqn:abcEquiv} is solvable in $s\in\mathbb{Z}$, then there exists an integer~$s$ such that $$\sum_{k=0}^{n-1}f(k)\zeta^{ks}=0$$ and  condition (ii) is satisfied. For example, we can choose $a=b=1$ and
\begin{eqnarray}
c=\begin{cases}
		\frac{n-1}{2}, & n\equiv 1 \pmod 4 \\
		\frac{n+1}{2}, & n\equiv 3 \pmod 4 
	\end{cases}.
\end{eqnarray}
Then~\eqref{eqn:abcEquiv} is solved by
\begin{eqnarray}
s=\begin{cases}
		\frac{n+3}{4}, &  n\equiv 1 \pmod 4 \\
		\frac{n+1}{4}, & n\equiv 3 \pmod 4 
	\end{cases}.
\end{eqnarray}

\section{Another Proof of Theorems~\ref{thm:1} and~\ref{thm:2}}\label{sec:integerEigen}

Now we proceed to give a proof of Theorems~\ref{thm:1} and~\ref{thm:2} that is independent of the eigenvector-eigenvalue equality~\eqref{eqn:EigenEqual}.

Let $A$ denote the $(n-1)\times (n-1)$ Hermitian matrix with diagonal elements equal to zero and off-diagonal elements $(a_{ij})_{i\neq j}$ given by
\begin{eqnarray}\nn
\frac{1}{1-x_{i-j}}, \quad\quad 1\leq i\neq j\le n-1,
\end{eqnarray}
where
\begin{eqnarray}\nn
x_k=
\zeta^k, & \forall~k.
\end{eqnarray}
Clearly the left-hand side of \eqref{eqn:odd1} is equal to $\det(A)$. It is convenient to multiply~$A$ from the right by the diagonal matrix
$B_s$ whose diagonal entries are given as
\begin{eqnarray}\nn
1-x_{is}, \quad\quad 1\leq i\le n-1,
\end{eqnarray}
by fixing any $s\in\{-\frac{n-1}{2},-\frac{n-3}{2},\cdots,-1,1,2,\cdots,\frac{n-1}{2}\}$.
Define $C_s\defeq AB_s$. If $n$ is prime or $s=1$, we have
\begin{eqnarray}\label{eqn:keyEq}
\det(C_s)=\det(AB_s)=\det(A)\det(B_s)=n\det(A).
\end{eqnarray}
To see this, note that if $n$ is prime, then $\{is\}_{0\le i\le n-1}$ form a complete residue system modulo~$n$ for any fixed~$s\in\{-\frac{n-1}{2},-\frac{n-3}{2},\cdots,-1,1,2,\cdots,\frac{n-1}{2}\}$. Therefore, $\{1-x_{is}\}_{1\le i\le n-1}\bigcup\{0\}$ form exactly the complete solution set of the following equation:
\begin{eqnarray}\label{eqn:simplePoly1}
(1-x)^n=1.
\end{eqnarray}
After cancelling $1$ and the trivial term $x$ (for root~$0$) from both sides, the remaining equation (note that~$n$ is odd) $x^{n-1}+\cdots+n=0$ has solutions exactly given by $\{1-x_{is}\}_{1\le i\le n-1}$ and their product equals the constant~$n$ in the left-hand side of the equation. Equivalently, we have 
\begin{eqnarray}
\det{B_s} = \prod_{1\le i\le n-1}(1-x_{is}) = n.
\end{eqnarray}
Note that if $n$ is prime, then $(x_{is})_{1\le i\le n-1}$ is a permutation of $(x_i)_{1\le i\le n-1}$ for any fixed $s\in\left\{-\frac{n-1}{2},\cdots,-1,1,\cdots,\frac{n-1}{2}\right\}$. In this case $is_1\not\equiv is_2~\pmod{n}$ if $s_1\not\equiv s_2~\pmod{n}$. For $s=1$ but~$n$ not necessarily prime, we note that $\{1-x_{is}\}_{1\le i\le n-1}\bigcup\{0\}$ still form exactly the complete solution set of~\eqref{eqn:simplePoly1} since~$\zeta$ is primitive. So $\det{B_1}=n$ by the argument above. 

For general odd $n>1$, identity~\eqref{eqn:odd1} is equivalent to the following identity:
\begin{eqnarray}\label{eqn:keyEq1}
\det(C_1)=\det(AB_1)=\det(A)\det(B_1)=(-1)^{\frac{n-1}{2}}\left(\frac{n-1}{2}!\right)^2.
\end{eqnarray}

To analyze~\eqref{eqn:keyEq1}, we use \eqref{eqn:l1} in Lemma~\ref{lem:lm1} to obtain the following identity for any fixed $s\in\left\{0,1,2,\cdots,\frac{n-3}{2},\frac{n+1}{2},\cdots,n-1\right\}$ and $k\in\{1,2,\cdots,n-1\}$:
\begin{eqnarray}
\sum^{n-1}_{j=1,j\neq k}\frac{1-x_{j(\frac{n-1}{2}-s)}}{1-x_{k-j}}x_{js}&=&\sum^{n-1}_{j=1,j\neq k}\frac{x_{js}-x_{j\frac{n-1}{2}}}{1-x_{k-j}}\\
&=&\left(\frac{n-1}{2}-s\right)x_{ks}.
\end{eqnarray}
Note that the above equality also holds for $s=\frac{n-1}{2}$, but we do not need this fact. It shows that $s$ is an eigenvalue of $C_{s}$ for each $s\in\left\{-\frac{n-1}{2},\cdots,-1,1,\cdots,\frac{n-1}{2}\right\}$. If we can show that $s$ is also an eigenvalue of $C_{1}$, then
\begin{eqnarray}\label{eqn:C1}
\det(C_1)=\prod_{-\frac{n-1}{2}\le s\le \frac{n-1}{2},s\neq0} s= (-1)^{\frac{n-1}{2}}\left(\frac{n-1}{2}!\right)^2.
\end{eqnarray}

To prove that $C_1$ has all the required integer eigenvalues, we construct a basis of the $(n-1)$-dimensional vector space on the complex field $\mathbb{C}$ by using the Fourier transform \cite{LiuPetrov}.
Denote by $u_j~(j=0,1,\cdots,n-1)$ the column-vector with coordinates $(x_{ji})_{1\le i\le n-1}$. Note that
\begin{eqnarray}\label{eqnn:FourierDep}
u_0+u_1+\cdots+u_{n-1}=0.
\end{eqnarray}
Furthermore, any $(n-1)$ vectors $u_j$'s form a $(n-1)*(n-1)$ minor of the orthogonal basis for the classical Fourier vector transform in the $n$-dimensional complex space $\mathbb{C}^n$. So the $n-1$ vectors are linearly independent and form a basis of the $(n-1)$-dimensional complex space $\mathbb{C}^{n-1}$, for otherwise the original full $n*n$ Fourier matrix would not have the full rank according to~\eqref{eqnn:FourierDep}. Let us consider the specific basis $\{u_0,u_1,\cdots,u_{n-1}\}\backslash\{u_{(n-1)/2}\}$. According to \eqref{eqn:l1} in Lemma~\ref{lem:lm1}, we have that
\begin{eqnarray}
C_1u_s=\begin{cases}
		\left(\frac{n-1}{2}-s\right)u_s-\left(\frac{n-1}{2}-s-1\right)u_{s+1}, & s\in\{0,1,\cdots,n-2\} \\
		-\frac{n-1}{2}u_{n-1}-\frac{n-1}{2}u_0, & s=n-1
	\end{cases}.
\end{eqnarray}
By ignoring the case of $s=\frac{n-1}{2}$ in the above, $C_1$ can be rewritten in the basis $\{u_0,u_1,\cdots,u_{n-1}\}\backslash\{u_{(n-1)/2}\}$ with the form
\begin{equation*}
\begin{pmatrix}
X & Y \\
0 & Z
\end{pmatrix},
\end{equation*}
where $X$ and $Z$ are $\frac{n-1}{2}\times\frac{n-1}{2}$  lower-triangular matrices with specific forms given by
\begin{equation*}
X=\begin{pmatrix}
\frac{n-1}{2} &  & & & &\\
-\frac{n-3}{2} & \frac{n-3}{2} & & & &\\
&-\frac{n-5}{2} & \frac{n-5}{2} & & &\\
& & \cdot & \cdot & &\\
& &  & \cdot & \cdot &\\
& & & &-1 & 1
\end{pmatrix}
\end{equation*}
and
\begin{equation*}
Z=\begin{pmatrix}
-1 &  & & & &\\
2 & -2 & & & &\\
&3 & -3 & & &\\
& & \cdot & \cdot & &\\
& &  & \cdot & \cdot &\\
& & & &0 & -\frac{n-1}{2}
\end{pmatrix}.
\end{equation*}
Based on the above forms and the fact that similar matrices share the same set of eigenvalues, the eigenvalues of $C_1$ are those of $X$ and $Z$ combined, which are exactly $\{-\frac{n-1}{2},\cdots,-1,1,\cdots,\frac{n-1}{2}\}$. We point out that the elements of~$Y$ are all zero except the top-right corner equal to $-\frac{n-1}{2}$. 

Now the proof of~\eqref{eqn:odd2} follows a similar process as above by applying \eqref{eqn:l2} in Lemma~\ref{lem:lm1} to the corresponding matrix $C'_1=A'B_1$ and considering the basis $\{u_1,\cdots,u_{n-1}\}$. Specifically, Let~$A'$ denote the $(n-1)\times (n-1)$ Hermitian matrix with diagonal elements equal to zero and off-diagonal elements $(a_{ij})_{i\neq j}$ given by
\begin{eqnarray}\nn
\frac{1+x_{i-j}}{1-x_{i-j}}, \quad\quad 1\leq i\neq j\le n-1.
\end{eqnarray}
Clearly the left-hand side of \eqref{eqn:odd2} is equal to $\det(A')$. Furthermore, by applying \eqref{eqn:l2} in Lemma~\ref{lem:lm1}, we have that
\begin{eqnarray}
C_1'u_s=\begin{cases}
		(n-2s)u_s + (2s+2-n)u_{s+1}, & s\in\{1,\cdots,n-2\} \\
                     (2-n)u_{n-1}, & s=n-1\\
		(1-n)u_1, & s=0
	\end{cases}.
\end{eqnarray}
By ignoring the case of $s=0$ in the above, $C_1'$ can be rewritten in the basis $\{u_1,\cdots,u_{n-1}\}$ with the form
\begin{equation*}
Z=\begin{pmatrix}
n-2 &  & & & & & & & &\\
4-n & n-4 & & & & & & & &\\
& \cdot& \cdot & & & & & & &\\
& & \cdot& \cdot & & & & & &\\
& &  &  -1& 1 & & & & &\\
& & & & 1 & -1 & & & &\\
& & & & & \cdot& \cdot & & &\\
& & & & & & \cdot& \cdot & &\\
& & & & & & & \cdot &4-n &\\
& & & & & & & & n-2&2-n
\end{pmatrix}.
\end{equation*}
So the eigenvalues of $C_1'$ are exactly $\{-(n-2), -(n-4), \cdots, -1, 1, \cdots, n-4, n-2\}$ and~\eqref{eqn:odd2} is proved immediately.

\section{Appendix}
The following lemma is well-known (see, \eg \cite{Gessel,CalogeroPerelomov,Sun6}) and we give an induction proof here for easy connections to the proof in Sec.~\ref{sec:integerEigen}.
\begin{lem}\label{lem:lm1}
Define $x_k=\zeta^k$ where $k\in\mathbb{Z}$ and~$\zeta$ is a primitive $n$-th root of unity for $n>1$. For any integers $s\in\{0,1,\cdots,n-1\}$ and $k\in\{1,2,\cdots,n-1\}$, we have
\begin{eqnarray}\label{eqn:l1}
\sum^{n-1}_{j=1,j\neq k}\frac{x_{js}}{1-x_{k-j}}&=&\left(\frac{n-1}{2}-s\right)x_{ks}-\frac{1}{1-x_k},\\\label{eqn:l2}
\sum^{n-1}_{j=1,j\neq k}\frac{x_{js}}{1-x_{j-k}}&=&
\begin{cases}
		\left(s-\frac{n+1}{2}\right)x_{ks}-\frac{1}{1-x_{-k}}, & s>0 \\
		\frac{n-1}{2}-\frac{1}{1-x_{-k}}, & s=0
	\end{cases}.
\end{eqnarray}
\end{lem}
\begin{proof}
\begin{eqnarray}
\sum^{n-1}_{j=1,j\neq k}\frac{x_{js}}{1-x_{k-j}}&=&x_{ks}\sum^{n-1}_{j=1,j\neq k}\frac{x_{(j-k)s}}{1-x_{k-j}}\\
&=&x_{ks}\left[\sum^{n-1}_{j=1}\frac{x_{js}}{1-x_{-j}}-\frac{x_{-ks}}{1-x_k}\right].
\end{eqnarray}
Note that
\begin{eqnarray}
\sum^{n-1}_{j=1}\frac{x_{js}}{1-x_{-j}}&=&\sum^{n-1}_{j=1}\left[\frac{x_{j(s-1)}}{1-x_{-j}}+x_{js}\right]\\
&=&-1+\sum^{n-1}_{j=1}\frac{x_{j(s-1)}}{1-x_{-j}}\\
&=&\cdots\\
&=&-s+\sum^{n-1}_{j=1}\frac{1}{1-x_{-j}}\\\label{eqn:s1}
&=& \frac{n-1}{2}-s.
\end{eqnarray}
The last equality is obtained by comparing the coefficient of $x^{n-2}$ in the following polynomial equation after cancelling the trivial term $x^n$ and dividing by $n$ on both sides:
\begin{eqnarray}
\left(1-\frac{1}{x}\right)^n=1.
\end{eqnarray}
The case of $s=0$ is thus trivial in \eqref{eqn:l2}. For $s>0$, we only need to observe that
\begin{eqnarray}
\sum^{n-1}_{j=1}\frac{x_{js}}{1-x_{j}}&=&\sum^{n-1}_{j=1}\left[\frac{x_{j(s-1)}}{1-x_{j}}-x_{j(s-1)}\right]\\
&=&1+\sum^{n-1}_{j=1}\frac{x_{j(s-1)}}{1-x_{j}}\\
&=&\cdots\\
&=&s-1+\sum^{n-1}_{j=1}\left[\frac{1}{1-x_{j}}-1\right]\\\label{eqn:s2}
&=&s-\frac{n+1}{2}.\label{eqn:eigenV1}
\end{eqnarray}
\end{proof}

\vspace{2em}

\section*{Declarations}
\begin{itemize}
    \item {\bf Ethics approval and consent to participate} Not applicable
    \item{\bf Consent for publication} Yes
     \item {\bf Availability of data and material} Available upon request
    \item{\bf Competing interests} Not applicable
     \item {\bf Funding} Not applicable
    \item{\bf Authors' contributions} The paper has a sole author responsible for all contributions.
     \item {\bf Data Availability} Available upon request
    \item{\bf Acknowledgements} The author is deeply indebted to Prof. Fedor Petrov from
Saint Petersburg State University for some key discussions on Fourier analysis in circulant matrices and to Prof. Xuejun Guo and Prof. Zhi-Wei Sun from Nanjing University for some very helpful suggestions. We also thank the anonymous reviewers for the highly valuable comments that helped improve this paper.
\end{itemize}

\vspace{2em}



\begin{thebibliography}{0}

\bibitem{CalogeroPerelomov} F. Calogero and A.M. Perelomov, {\sl Some Diophantine Relations Involving Circular Functions of Rational Angles}, {Lin. Alg. Appl.} 25, 91-94 (1979).

\bibitem{DentonEtal} P. B. Denton, S. J. Parke, T. Tao, and X. Zhang, {\sl Eigenvectors from eigenvalues: a survey of a basic identity in linear algebra}, Bulletin of the American Mathematical Society 59(1), 31-58 (2022).

\bibitem{Gessel} I. Gessel, {\sl Generating Functions and Generalized Dedekind Sums}, {Electronic J. Combin.} 4(2), 137-153 (1997).

\bibitem{Guo} X. Guo, X. Li, Z. Tao, and T. Wei, {\sl The Eigenvectors-Eigenvalues Identity and Sun's Conjectures on Determinants and Permanents}, {Linear Multilinear A.} 72(7), 1071–1077 (2022).

\bibitem{LiuPetrov} F. Petrov, {\sl Answer to Integer Eigenvalues of a Class of Matrices Inspired by Prof. Zhi-Wei Sun's Conjecture}, {https://mathoverflow.net/q/424236,} June, 2022.

\bibitem{Petrov} F. Petrov, {\sl Email Communications}, June, 2022.

\bibitem{Sun1} Z.-W. Sun, {\sl On Some Determinants with Legendre Symbol Entries}, {Finite Fields Appl.} 56, 285-307 (2019).

\bibitem{Sun2} Z.-W. Sun, {\sl On Some Determinants Involving the Tangent Function}, {Ramanujan J.} 64, 309–332 (2024).

\bibitem{Sun3} Z.-W. Sun, {\sl On Permutations of ${1,\ldots,n}$ and related topics}, { J. Algebraic Combin.} 54, 893–912, (2021).

\bibitem{Sun4} Z.-W. Sun, {\sl Arithmetic Properties of Some Permanents}, {preprint,} arXiv:2108.07723.

\bibitem{Sun5} Z.-W. Sun, {\sl New Conjectures in Number Theory and Combinatorics} (in Chinese), Harbin Institute of
Technology Press, Harbin (2021).

\bibitem{Sun6} H. Wang and Z.-W. Sun, {\sl Proof of a Conjecture Involving Derangements and Roots of Unity}, {Electronic J. Combin.} 30(2), 1-10 (2023).





\end{thebibliography}
\end{document}